\documentclass[12pt]{article}
\usepackage{amsfonts}
\usepackage{amssymb}
\usepackage{amsthm}
\usepackage{amsmath}
\usepackage{xcolor}
\usepackage{tikz}
\DeclareRobustCommand
\Compactcdots{\mathinner{\cdotp\mkern-2mu\cdotp\mkern-2mu\cdotp}}
\def\O{\Omega}
\def\dist{\mathrm{dist}}
\def\A{\mathfrak{A}}
\def\C{\mathbf{C}}
\def\bC{\mathbf{\overline{C}}}
\def\R{\mathbf{R}}
\def\Z{\mathbf{Z}}

\def\Rea{\mathrm{Re}\, }
\def\Ima{\mathrm{Im}\, }

\numberwithin{equation}{section}
\newtheorem{theorem}{Theorem}[section]
\newtheorem{lemma}[theorem]{Lemma}
\newtheorem{proposition}[theorem]{Proposition}
\theoremstyle{remark}
\newtheorem*{ack}{Acknowledgments}
\newtheorem{remark}[theorem]{Remark}
\newtheorem{example}[theorem]{Example}
\title{On conformal metrics of constant positive curvature in the plane}
\author{Walter Bergweiler, Alexandre Eremenko and James Langley
}
\date{}
\begin{document}
\maketitle
\begin{center}
\emph{Dedicated to Professor Vladimir A.\ Marchenko \\
on the occasion of his 100th birthday}
\end{center}
\begin{abstract}
We prove three theorems about solutions of $\Delta u+e^{2u}=0$ in the plane.
The first two describe explicitly all concave and quasiconcave solutions.
The third theorem says that the diameter of the plane with respect
to the metric with line element $e^{u}|dz|$ is at least $4\pi/3$,
except for two explicitly described families of solutions $u$.

\smallskip
2010 MSC 35B99, 35G20, 30D15.

\smallskip
Keywords: Liouville equation, positive curvature, meromorphic function,
spherical derivative.

\end{abstract}

\section{Introduction}

The general solution of the differential equation
\begin{equation}\label{0}
\Delta u+e^{2u}=0
\end{equation}
in a simply connected region in the plane was written by Liouville as
\begin{equation}\label{00}
u=\log\frac{2|f'|}{1+|f|^2},
\end{equation}
where $f$ is a meromorphic local homeomorphism, that is a meromorphic function
with only simple poles which satisfies $f'(z)\neq 0$. The geometric interpretation is
that the metric $\sigma$ with the line element 
$$\frac{2|f'(z)|}{1+|f(z)|^2}|dz|$$
is the pull-back of the standard metric on the unit sphere via $f$.
Here $f$ is called the developing map of the metric,
and the relation \eqref{00} will be assumed throughout the paper.

Expression (\ref{00}) 
is due to J. Liouville \cite{Li0,Li},
though the equivalent result that every two metrics of
the same constant curvature are locally isometric is contained
in the earlier paper of F. Minding \cite{M}.
Formula \eqref{00} for the general solution of \eqref{0} is widely used
in modern literature, see, for example \cite{02,Lin}.

In this paper, we discuss equation \eqref{0} in the plane.

In \cite[Thm.\ 1.6]{1}, solutions of \eqref{0} which are bounded from
above are completely described: they are exactly those for which
$f$ in (\ref{00}) is either linear-fractional or 
of the form 
\begin{equation}\label{20}
f(z)=L(e^{az+b}),
\end{equation}
where $L$ is a linear-fractional transformation 
and $a,b\in\C$, $a\neq 0$. It was noticed in \cite{1} that all solutions
of \eqref{0} with developing map of the form \eqref{20} are one-dimensional:
after a complex affine change of the variable,
they depend on one real variable only.

In this paper we prove two conjectures stated in \cite{1}.
Our first theorem proves Conjecture~2 of that paper.
\begin{theorem}
\label{thm1}
If $u$ is concave, then $f$ is of the form \eqref{20}.
\end{theorem}

This was proved in \cite{1} under the additional condition that $u$ is bounded
from above.  
In Section \ref{firstproof}, we give a  direct proof 
 of Theorem \ref{thm1}, but an alternative approach via differential equations
  delivers a stronger conclusion when
 $f$ is transcendental. A function $h\colon \C \to \R$ is called \textit{quasiconcave} if, for any 
 $a_1, a_2 \in \C$, we have $h(z) \geq \min \{ h(a_1), h(a_2) \}$ on the line segment from
 $a_1$ to $a_2$: this is equivalent to the condition that, for
 every $c \in \R$, the set $\{ z \colon h(z) \geq c \}$ is convex.  If $f$ is linear-fractional then via a rotation of the Riemann sphere it may be assumed that $f(\infty) = \infty$, so that each set 
 $\{ z \colon u(z) \geq c \}$ is  a  disk and $u$ is quasiconcave. 
 
\begin{theorem}
\label{thm1a}
If $u$ is quasiconcave, then $f$ is linear-fractional or of the form \eqref{20}.
Indeed, for transcendental $f$ not of the form \eqref{20}, 
and for any $M > 0$, there exist $a_1, a_2 \in \C$  such that 
\begin{equation}
u\left(\frac{a_1+a_2}2\right) < \min\{ u(a_1), u(a_2) \}  - M.
\label{notconcave}
\end{equation}
\end{theorem}
If $f$ is linear-fractional or of the form $f=\phi(e^{az+b})$,
where $\phi$ is a rotation of the sphere, then the diameter 
of the plane with respect to the metric $\sigma$ is~$\pi$.
It was conjectured in~\cite[Question 8.1]{GL}, \cite[Conjecture~1]{1} that the diameter 
is strictly greater than $\pi$ otherwise. We shall prove the following stronger result.
\begin{theorem}
\label{thm2}
The diameter of the plane with respect to the metric $\sigma$
is at least $4\pi/3$ unless $f$ is linear-fractional or
of the form \eqref{20}.
\end{theorem}
It is proved in \cite[Cor.\ 3.4]{1} that if $f(z)=e^z+t$, 
then the diameter is equal to $\pi+2\arctan |t|$.
So it can be any number in $[\pi,2\pi)$.

We do not know whether the estimate $4\pi/3$ in Theorem \ref{thm2} is
the best possible. At the end of the paper we give an example
of a metric $\sigma$ such that the plane with this metric
has infinite diameter. 
This example answers another question asked in \cite[Question 8.1]{GL}.

\section{Proof of Theorem \ref{thm1}}\label{firstproof}

Since the case when $u$ is bounded from above
has been treated in \cite{1}, we may assume that $u$ is unbounded from above.
We distinguish the cases whether $f$ in \eqref{00} has finite or infinite order,
see~\cite[Section 2.1]{Goldberg2008} for the definition of the 
order of a meromorphic function.
The asymptotic formula  (1.6) in \cite[Thm.\ 1.9]{1} shows that
$u$ cannot be concave when $f$ is of finite order, unless
$f$ is of the form (\ref{20});
for further details in this case, see Section \ref{secondproof}. 

Thus we limit ourselves to the case that $f$ is of infinite order.
Let us call a point $a\in\bC$ exceptional if $f(z)\to a$ as $z\to\infty$
uniformly with respect to $\arg z$ in a sector of opening
$\pi/3$. Since there can be at most 6 exceptional points,
we can apply a rotation of the sphere to $f$ to ensure that $\infty$ is
not exceptional.

Then $f'$ is a meromorphic function of infinite order without zeros,
so $f'=1/w$, where $w$ is entire of infinite order, that is
\begin{equation}\label{infi}
\limsup_{z\to\infty}\frac{\log\log|w(z)|}{\log|z|}=\infty.
\end{equation}
Consider the sets
$$E=\{ z\colon u(z)\geq 0\}\quad \mbox{and}\quad D=\C\backslash E.$$
Since $u$ is concave, $E$ is convex. Since $u$ is unbounded from above,
$E$ is unbounded. Let us assume without loss of generality that
$0\in E$ and 
\[
u(0)>0.
\]
This can be achieved by translation of the independent
variable. Since $E$ is unbounded, closed and convex, and contains $0$, there is at least one ray 
$$\ell_\theta=\{ z=t e^{i\theta}\colon t\geq 0\}$$
contained in $E$.
Let $I$ be the set of arguments $\theta\in \R/(2\pi\Z)$ of the rays $\ell_\theta$ which 
are contained in $E$.
Unless $E$ is a parallel strip, in which case $I$ consists of two points, 
$I$ is a closed interval of length at most $\pi$.

Let us call $\theta_0\in\R/(2\pi\Z)$ a {\em direction of fast decrease}
if there is a sequence 
$(z_n)$ tending to $\infty$ 
such that
\begin{equation}\label{theta0}
\lim_{n\to\infty}\frac{\log^+(-u(z_n))}{\log|z_n|}=\infty\quad\mbox{and}\quad
\arg z_n\to\theta_0.
\end{equation}
Since 
\begin{equation}
\label{ufw}
u(z)\leq\log 2|f'(z)|=-\log|w(z)|+\log 2,
\end{equation}
(\ref{infi}) implies that there exists a direction of fast decrease.

We claim that there exists a direction of fast decrease which does not
belong to $I$. Clearly an interior point of $I$ cannot be a direction of
fast decrease, so the claim will follow if we show that there are more than
two directions of fast decrease. Let $(z_n)$ be a sequence which satisfies
(\ref{theta0}), and let $\theta_0\in(\theta_1,\theta_2)$ where
$|\theta_1-\theta_2|<\pi$. Draw a segment $[a_n,b_n]$ through $z_n$ 
such that $\arg a_n=\theta_1$ and $\arg b_n=\theta_2$.
By concavity, the minimum of $u$ on $[a_n,b_n]$ is attained either
at $a_n$ or at $b_n$, so one of $\theta_1$ and $\theta_2$ is also
a direction of fast decrease. Since $\theta_1$ and $\theta_2$
can be chosen in many ways there are many directions of fast decrease.

We conclude that there is a half-plane $H\subset D$  
which contains a ray in a direction of fast decrease.

By rotating the independent variable, we assume that $H$ is a left
half-plane, say $H=\{ z\colon \Rea z<- c\}$ where $c>0$.

Considering the restriction of $u$ to the intervals $[-c+iy,0]$, we see that
the derivatives $(d/dt)u((-c+iy)t)$ are negative somewhere in these intervals.
The concavity of $u$ implies that for every $\epsilon>0$ we have
\begin{equation}\label{est}
u(re^{i\theta})\leq -kr,\quad |\theta-\pi|\leq \frac{\pi}{2}-\epsilon,\quad r>r_0,
\end{equation}
where $k$ and $r_0$ are some positive constants that depend on $\epsilon$.
We denote the angular sector in (\ref{est}) by $A$, and fix $\epsilon$
so that $A$ contains a ray of fast decrease.

Consider now the set $G=\{ z\colon |w(z)|>2\}$. 
By~\eqref{ufw} we have $G\subset D$.
On the other hand, we claim that 
\begin{equation}\label{es}
\log|w(z)|\geq -u(z)-C, \quad z\in A,
\end{equation} 
if $r_0$ is large enough. Here $C$ is a positive constant.
To prove (\ref{es}) we first notice that for every sequence $(z_n)$ in $A$, tending
to infinity, the sequence $(f(z_n))$ on the Riemann sphere is a Cauchy sequence
(with respect to the spherical metric). This follows from
the estimate of the spherical distance
$$\dist(f(z_n),f(z_m))\leq\int_{z_m}^{z_n}e^{u(z)}|dz|,$$
and the estimate (\ref{est}). Therefore there exists $ a\in\bC$ such that
$f(z)\to a$ as $z\to\infty$, $z\in A$. This limit $a$ is an exceptional
point as defined in the beginning of the proof, and thus $a\neq\infty$.
Using (\ref{est}) we see that
$$\dist(a,f(z))\leq\int_{-\infty}^ze^{u(z)}|dz|\leq\frac{1}{k}e^{-kr_0}.$$ 
Choosing $r_0$ large enough, we achieve that there exists a constant $C_1$ such that
 $|f(z)|\leq C_1$ for $z\in A$. Thus 
$$u(z)\geq \log(2|f'(z)|)-\log(1+C_1^2) =-\log |w(z)|+\log 2-\log(1+C_1^2)$$
for $z\in A$ and 
we obtain (\ref{es}).

Since $u(z)\to -\infty$ in $A$ we conclude from (\ref{es})
that $\log|w|$ is bounded from below in $A$,
say $\log|w(z)|\geq C_2$ for $z\in A$.
Thus $v=\log|w|-C_2$ is a positive harmonic function in~$A$.
It follows (see \cite[p.~87]{L}) that in any proper subsector
of $A$, the function $v$ cannot grow faster than a power.
In particular, for a sequence $(z_n)$ as in \eqref{theta0},
where $\theta_0$ is a direction of fast decrease with $\ell_{\theta_0}\subset A$,
we have $v(z_n)\leq |z_n|^{C_3}$ for some $C_3>0$.
Together with \eqref{es} this yields that
\[
-u(z_n)\leq \log|w(z_n)|+C= v(z_n)+C_2+C\leq |z_n|^{C_3}+C_2+C,
\]
contradicting~\eqref{theta0}.

\section{Proof of Theorem \ref{thm1a}}\label{secondproof}
Assume that $f$ is transcendental but not of the form \eqref{20}.
Then the Schwarzian $2A$ of $f$ is a
non-constant entire function \cite{Hil2,schwarzian} and
\begin{equation}
f = \frac{f_1}{f_2} ,
\label{1}
\end{equation}
where $f_1, f_2$ are linearly independent solutions of
\begin{equation}
w''+Aw = 0.
\label{2}
\end{equation}

Suppose first that $A$ is a polynomial of degree $d > 0$. Then by the classical theory of asymptotic integration \cite{1,Hil2}
 there are $d+2$ equally spaced Stokes rays which divide the plane into open sectors, on each of which $f(z)$ tends to some asymptotic value, these values being different on adjacent sectors. Thus
 by a rotation of the independent variable 
 it may be assumed that $f(z)$ tends to a finite asymptotic   value on a sector of opening $2\pi/(d+2)$,
 symmetric about the positive real axis. 
 Hence the sectorial asymptotics for \eqref{2} give a constant $c > 0$ with the following property.
Let $M, \delta$ be positive constants with $\delta $ small: then 
$$
u(re^{i\theta}) = - c r^{(d+2)/2} \cos \left( \frac{(d+2)\theta}2 \right) + o \left( r^{(d+2)/2} \right) 
$$
as $r \to \infty$, uniformly for real $\theta$ with $| \theta | \leq  \pi/(d+2) - \delta $ (see \cite{Hil2} and
\cite[Thm.~1.9, formula (1.6)]{1}). Since $\delta$ is small it is then clear that 
\eqref{notconcave} holds with $a_j = r \exp( (-1)^j i (\pi/(d+2) - \delta) )$ and $r$ sufficiently large. 

Assume henceforth that $A$ is transcendental. The  proof will use estimates, analogous to the sectorial asymptotics
in the polynomial case,
which were proved in  \cite{La5,schwarzian,blnewqc}
for solutions of (\ref{2}) on a neighborhood 
of a  maximum modulus point of $A$.
Let $N(r) $ be
the central index of $A$, and take $\phi(r) = N(r)^{1/3} $ in the notation of  \cite[Sections 2, 3]{schwarzian}.
Let $r > 0$ be large and lie outside the  exceptional set $E_0$ arising from applying the Wiman-Valiron theory
\cite{Hay5}  to  $A$,
and take $z_r$ with $|z_r| = r$ and $|A(z_r) | = M(r, A)$. Then
\begin{equation}
 \phi(r) = N(r)^{1/3} = o( \log |A(z_r) | )
 \label{a0}
 \end{equation}
 and
 \begin{equation}
A(z) \sim \left( \frac{z}{z_r} \right)^N A(z_r) , \quad
\frac{A'(z)}{A(z)} \sim \frac{N}z, \quad \frac{A''(z)}{A(z)} \sim \frac{N^2}{z^2} ,
\quad N = N(r) ,
\label{a1}
\end{equation}
for $z \in D(z_r, 8)$, where $D(z_r, L)$ denotes the logarithmic rectangle
\begin{equation}
D(z_r, L)
=
 \{ z_r e^\tau \colon | {\rm Re} \, \tau | \leq L N(r)^{-2/3}  , \quad
| {\rm Im } \, \tau | \leq L N(r)^{-2/3} \}.
\label{a2}
\end{equation}
Let $w_r = z_r \exp( -4 N(r)^{-2/3} )$; then \cite[formula (10)]{schwarzian} gives, on  $D(z_r, 4)$,
\begin{equation}
\begin{aligned}
Z &= \frac{2w_r A(w_r)^{1/2}}{N+2} + \int_{w_r}^z A(t)^{1/2} \, dt \\
&\sim \frac{2z A(z)^{1/2}}{N+2}  \sim Z(z_r) \left( \frac{z}{z_r} \right)^{(N+2)/2} ,\\
\log \frac{Z(z)}{Z(z_r)} &= \frac{N+2}2 \, \log \frac{z}{z_r} + o(1), \\
\frac{d \log Z}{\log z} &= \frac{zA(z)^{1/2}}{Z} \sim \frac{N+2}2 .
\label{a3}
\end{aligned}
\end{equation}
As in the previous case, fix $M > 0$ and let $\delta $ be small and positive.
The following is a slightly stronger assertion than \cite[Lemma 3.1]{schwarzian}.
\begin{lemma}
\label{lemcover}
Let  $Q$ be a large positive integer,
 and let $r \not \in E_0$ be large. Then $\log Z$ is a univalent function of
$\log z$ on $D(z_r, 7/2)$ and there exist at least $Q$ pairwise disjoint simple islands $H_q$
in $D(z_r, 3)$ each mapped univalently by $Z$ onto the closed logarithmic rectangle
$$
J_1 = \left\{ Z\colon R \leq |Z| \leq S, \quad | \arg Z | \leq \pi - \delta \right\} ,
$$
in which
$$
R = |Z(z_r)| \exp \left( - N(r)^{1/3} \right), \quad S = |Z(z_r)| \exp \left(  N(r)^{1/3} \right) ,
$$
while $R$ and $S/R$ are both large.
\end{lemma}
\begin{proof} 
The first two assertions follow from (\ref{a3}), exactly
 as in \cite{schwarzian}. To see that $R$ is large, use (\ref{a3}) to write
$$
\log R \geq  \frac12 \log M(r, A) + \log r - \log (N+2) - N(r)^{1/3} - O(1),
$$
the right-hand side being large and positive by (\ref{a0}).
\end{proof} 

Next, for $z$ in some $H_q$, apply to (\ref{2}) the Liouville transformation \cite{Hil2}
\begin{equation}
\label{liouville}
W(Z) = A(z)^{1/4} w(z),
\end{equation}
so that
$W$ satisfies
\begin{equation}
\frac{d^2W}{dZ^2} + (1-F_0(Z))W = 0,
\label{a7}
\end{equation}
in which, by (\ref{a1}),
$$
F_0(Z) =
\frac{A''(z)}{4 A(z)^2} - \frac{5 A'(z)^2}{16 A(z)^3} , \quad |F_0(Z)| \leq \frac3{|Z|^2} \quad \hbox{on $J_1$}.
$$
Then \cite[Lemma 2.1]{blnewqc} gives solutions $U_1(Z), U_2(Z)$ of (\ref{a7}) which satisfy
\begin{equation}
\label{a9}
U_1(Z) \sim e^{-iZ}, \quad U_2(Z) \sim e^{iZ}, \quad W(U_1, U_2) \sim 2i ,
\end{equation}
uniformly for $Z$ in the set
\begin{equation}
\label{a10}
J_2 = J_1 \setminus \left\{ Z\colon  {\rm Re} \, Z < 0, \quad |{\rm Im} \, Z | < R  \right\} .
\end{equation}
The restriction to $J_2$ is a consequence of the method of proof, which requires removal of
the ``shadow'' of the disk $B(0, R)$ \cite{Hil2}.
Hence (\ref{liouville}) and (\ref{a9}) deliver solutions $u_1, u_2$ of (\ref{2}) satisfying
\begin{equation}
\label{a11}
u_1(z) \sim A(z)^{-1/4} e^{-iZ}, \quad u_2(z) \sim  A(z)^{-1/4} e^{iZ}, \quad W(u_1, u_2) \sim 2i ,
\end{equation}
on the preimage $H_q'$ of $J_2$ in $H_q$.

\subsection{Two line segments lying in the same $H_q'$}

The next step is to choose two line segments, both  lying  in the same $H_q'$, one of
which will be used to show that $u $ is not concave. By (\ref{a3}) there exist $p \in \R$ and $z_r'$ satisfying
\begin{equation}
- \frac{4\pi}{N+2} < p < \frac{4\pi}{N+2} ,  \quad z_r' = z_r e^{ip}, \quad  |Z(z_r')| \sim |Z(z_r)|, \quad \arg Z(z_r') = 0 .
\label{c2}
\end{equation}
Now set
\begin{equation}
\begin{aligned}
\zeta_1^+ &= z_r' e^{i 4 \delta/(N+2) }, \quad
\zeta_2^+  = z_r' e^{i 2( \pi - 2 \delta)/(N+2) }, \\ 
\zeta_1^- &= z_r' e^{ - i 4 \delta/(N+2) }, \quad
\zeta_2^- = z_r' e^{- i 2( \pi - 2 \delta)/(N+2) }.
 \label{c3}
\end{aligned}
\end{equation}
Let $S^+$ be the line segment from $\zeta_1^+$ to $\zeta_2^+$,
let $S^-$ be that from $\zeta_1^-$ to $\zeta_2^-$, and let $\Sigma$ be the arc of the circle $|z| = r$ from $\zeta_1^+$ to $\zeta_1^-$ via $z_r'$: these  lie in $D(z_r, 3)$, by (\ref{a2}), (\ref{c2}) and the fact that
$N$ is large. Indeed, elementary trigonometry gives
\[
r \geq |z| \geq r \cos \left( \frac{\pi-4\delta}{N+2} \right) \geq r \left( 1 - O \left( \frac1{N^2} \right) \right)
\quad \hbox{for} \quad z \in S^+ \cup S^- ,
\]
from which it follows, in view of (\ref{a1}) and (\ref{a3}), that
\begin{equation}
|A(z)| \sim |A(z_r)| =M(r, A)  \quad \hbox{and} \quad
|Z(z)| \sim T = |Z(z_r)|
\label{c5}
\end{equation}
for $z \in S^+ \cup S^- \cup \Sigma $. On the other hand (\ref{a3}), (\ref{c2}) and (\ref{c3}) also yield
$$
 - \pi + 2  \delta + o(1) \leq
 \arg Z(z) \leq \pi - 2 \delta + o(1)
 \quad \hbox{for} \quad z \in S^+ \cup S^- \cup \Sigma .
$$
Combining this estimate with (\ref{a10}), (\ref{c5}) and the fact that $T/R = \sqrt{S/R}$ is large then shows that
$S^+, S^-$ indeed lie in the same $H_q'$.

Next,
let $\zeta_3^+$ be the midpoint of $S^+$, and $\zeta_3^-$ that of $S^-$. Then  (\ref{a3}), (\ref{c2}) and (\ref{c3}) deliver
\begin{equation}
\begin{array}{l}
\arg( Z(\zeta_1^\pm ))  = \pm 2  \delta + o(1) ,  \\
\arg( Z(\zeta_2^\pm )) = \pm ( \pi - 2  \delta ) + o(1),  \\
\arg( Z(\zeta_3^\pm  )) = \pm \frac{\pi}2 + o(1).
\end{array}
    \label{c6}
\end{equation}

\subsection{Estimates for $u$ }

On that $H_q'$ which contains $S^+ \cup S^-$, write
\begin{equation}
f_1 = C_1 u_1 + C_2 u_2, \quad f_2 = D_1 u_1 + D_2 u_2, \quad C_j, D_j \in \C  .
\label{a}
\end{equation}
Since the spherical derivatives of
$f$ and $1/f$ agree, it may be assumed that, among $C_1, C_2, D_1, D_2$, either
 $D_1$
or $ D_2$ has maximal modulus. Hence (\ref{1}) and (\ref{a}) yield constants $\alpha, \beta, \gamma$ with
\begin{equation}
\label{b1}
f = \frac{f_1}{f_2} = \frac{\alpha v + \beta}{\gamma v + 1} ,
\quad
\max \{ |\alpha|, |\beta| , |\gamma| \} \leq 1, \quad v = \left( \frac{u_1}{u_2} \right)^{\pm 1} .
\end{equation}

Suppose first that $D_1$ has maximal modulus. Then $v = u_2/u_1$ and, on $H_q'$, 
(\ref{a11}) delivers
\begin{equation}
\log |v| = -2 \, {\rm Im} \, Z + o(1) ,
\label{b1a}
\end{equation}
as well as
\begin{equation}
\begin{aligned}
f' &=  \frac{(\alpha-\beta\gamma) v' }{(\gamma v + 1 )^2} =  \frac{(\alpha-\beta\gamma)  }{(\gamma v + 1 )^2} \, \frac{W(u_1,u_2)}{u_1^2}  \\
&\sim \left(
\frac{\alpha-\beta\gamma}{(\gamma v + 1 )^2} \right)  \frac{2i}{u_1^2}
 \sim \left(
\frac{\alpha-\beta\gamma}{(\gamma v + 1 )^2} \right)  \, 2i A^{1/2} v .
\label{b3}
\end{aligned}
\end{equation}
Thus
(\ref{c5}), (\ref{c6}), and (\ref{b1a}) imply that
\begin{equation}
\log |v(\zeta_1^+)| \sim  \log |v(\zeta_2^+)| \sim  -2 T \sin (2 \delta ) ,
\quad
\log |v(\zeta_3^+)| \sim  -2 T .
\label{b2}
\end{equation}
It then follows from (\ref{b1})  that,  at the three points
$\zeta_1^+, \zeta_2^+ , \zeta_3^+$, both
$v$  and $f - \beta $ are small, so that,
$$
u( \zeta_j^+) 
= \log \frac{|\alpha - \beta \gamma|}{1+|\beta|^2} + \log 4
 + \frac12 \log M(r, A)  +  \log |v(\zeta_j^+)| + o(1) .
$$
Since $\delta$ is small,   while $T$ is large, (\ref{b2}) now implies that
\eqref{notconcave} holds with $a_j = \zeta_j^+$.
This completes the proof that $u$ is not quasiconcave,
provided that  $D_1$ has the largest modulus among $C_1, C_2, D_1, D_2$.
On the other hand, if $|D_2|$ is maximal then the same argument goes through using $v = u_1/u_2$ and the points
$\zeta_1^-, \zeta_2^-, \zeta_3^-$.
\phantom{}\hfill$\Box$

\section{Diameter of the plane with the metric $\sigma$. Proof of Theorem~\ref{thm2}}

Since the case of linear-fractional $f$ has been dealt with in \cite{1},
and a rational function of degree greater than $1$ cannot be locally univalent in $\C$,
we assume $f$ is transcendental.

We begin by recalling some notation and facts from the theory of
singularities of inverses of meromorphic functions; see, for example,
\cite{BE,0,01}.

Let $a$ be a point in $\bC$. We denote by $\A(a)$ the set of all open simply connected
neighborhoods of $a$. A (transcendental) {\em singularity} of $f^{-1}$ over $a$ is a function
$\O\mapsto V(\O)$ which assigns to every $\O\in\A(a)$
a connected component $V(\O)$ of the preimage $f^{-1}(\O)$ such that
$$\O_1\subset\O_2\Rightarrow V(\O_1)\subset V(\O_2),$$
and 
$$\bigcap_{\O\in\A}V(\O)=\emptyset.$$
A singularity over $a$ exists if and only if $a$ is an {\em asymptotic value};
that is, there exists a curve $\gamma\colon [0,1)\to\C$ such that
$\gamma(t)\to\infty$ and $f(\gamma(t))\to a$ as $t\to 1$.

If a region $D\subset\bC$ contains no asymptotic values, then the restriction 
of $f$ on any component of $f^{-1}(D)$ is a covering.

Singularities can be considered as elements
of a completion of $\C$ with respect to a
certain metric adapted to $f$, see \cite{0}. The open sets $V(\O)$ are called
{\em neighborhoods} of a singularity. A singularity is {\em isolated}
if it has a neighborhood $V(\O)$
which is not a neighborhood of any other singularity.
This condition implies that the restriction
\begin{equation}\label{cov}
f\colon V(\O)\backslash f^{-1}(a)\to\O\backslash\{ a\}
\end{equation}
is a covering. So it must be a universal covering, since we assumed
that $f$ has an essential singularity at $\infty$.
Singularities for which (\ref{cov}) is a universal covering are
called the {\em logarithmic singularities}
of $f^{-1}$ in the classical literature.

In the case when all singularities are isolated,
the metric completion of $(\C,\sigma)$
consists of $\C$ and one point for each singularity.
In the general case of non-isolated
singularities, the singularities can also be interpreted as elements of a
metric completion, but with respect to a different metric; see
\cite{0}. We will denote the metric completion of $(\C,\sigma)$ by
$\widetilde{\C}$. It will be used only in the case when
all singularities are isolated.

We will need several lemmas. 

\begin{lemma}
\label{lem41}
Let $f$ be a non-constant meromorphic function
and $\gamma$ a simple curve which consists of two asymptotic curves
with common starting point and distinct asymptotic values. Let $D$
be one of the two components of $\C\backslash\gamma$. Then
the restriction $f\vert_D$ has a dense image in $\bC$.
\end{lemma}

\begin{proof}
A theorem of Lindel\"of (see \cite[p.~179]{Goldberg2008} or \cite[p.~81]{L})
yields that $f\vert_D$ cannot be bounded.
Applying this to $1/(f-a)$ shows that $f\vert_D$ cannot omit a 
neighborhood of a point $a\in\C$.
\end{proof}

\begin{remark}
Using a deeper result of Heins \cite[Thm.~4]{H} one can actually show
that $f\vert_D$ omits at most two values in $\bC$.
\end{remark}

\begin{lemma}
\label{lem42}
Let $K\subset\bC$ be a geodesic arc of length $t\in (0,\pi)$: then
$$\sup_{z\in\bC}\dist(z,K)=\pi-\frac{t}{2}.$$
\end{lemma}

\begin{proof} 
Without loss of generality, we may assume that
\[
K=\{e^{i \phi}\colon - t/2 \leq \phi \leq t/2 \}.
\]
We claim that the maximal distance from $z$ to $K$ is attained for $z=-1$.
The conclusion follows from this, since the points in $K$ that are
closest to $z=-1$ are the points $e^{\pm it/2}$, implying that $\dist(-1,K)=\pi-t/2$.

Since the spherical distance
is a strictly increasing function of the chordal distance,
the above claim follows if we show that $z=-1$ has the maximal chordal distance from~$K$.
Let $\chi$ denote the chordal metric and $\dist_\chi(z,K)$ the distance 
from $z$ to $K$ with respect to this metric.

By symmetry, it is sufficient to
consider the case that $|z|\leq 1$ and $\Ima z\geq 1$ so that
$z$ has the form $z = re^{i \theta} $ with $0 \leq r \leq 1$ and $0 \leq \theta \leq \pi$. 
If $t/2<\theta\leq \pi$, then, 
using that $0<t<\pi$ and hence $\cos(\pi-t/2)<0$, we find that
\[
\begin{aligned}
\dist_\chi(z,K)^2
&\leq \chi(z,e^{it/2})^2
=\frac{4|z-e^{it/2}|^2}{(1+|z|^2)(1+|e^{it/2}|^2)} 
\\ &
=\frac{2(1+r^2-2r\cos (\theta-t/2))}{1+r^2} 
=2- \frac{4r}{1+r^2} \cos(\theta-t/2)
\\ &
\leq 
2- \frac{4r}{1+r^2} \cos(\pi-t/2)
\leq 
2- 2 \cos(\pi-t/2)
\\ &
=\chi(-1,e^{it/2})^2
=\dist_\chi(-1,K)^2 .
\end{aligned}
\]
On the other hand, if $0\leq\theta\leq t/2$, then
\[
\dist_\chi(z,K)^2
\leq \chi(z,e^{i\theta})
=2- \frac{4r}{1+r^2}\leq 2
< \dist_\chi(-1,K)^2 .
\]
It follows that the maximum of $\dist_\chi(z,K)$ and hence of $\dist(z,K)$ is attained for $z=-1$,
as claimed above.
\end{proof}

Our proof of Theorem \ref{thm2} splits into two parts:

\begin{proposition}
\label{prop41}
If $f$ has a non-isolated singularity, then
the diameter of $(\C,\sigma)$ is at least $2\pi$.
\end{proposition}

\begin{proof}
By rotating the sphere we may assume that there is a non-isolated
singularity over $0$. Let $D_\epsilon$ be the open spherical disk
of radius $\epsilon>0$ about zero. Then $V(D_\epsilon)$ must be a neighborhood
of some other singularity. We claim that this other singularity can be chosen with
asymptotic value $b\neq 0$. Indeed, if all singularities of which $V(D_\epsilon)$
is a neighborhood were to lie over $0$, then $f\colon V(D_\epsilon)\to D\backslash\{0\}$
would be a covering, contradicting our assumption that the singularity over $0$
is not isolated.

Thus there exists a curve $\gamma \subseteq V(D_\epsilon)$, both ends of which  tend to
$\infty$ in $\C$, and which is asymptotic for two  values, that is,
$f(z)\to 0$ and $f(z)\to b$ as $z\to\infty$ along the two ends of $\gamma$.
By removing any loops we may assume that $\gamma$ is simple.
So it divides the plane into two parts, which we denote by $D_1$ and $D_2$.
By construction, $f(\gamma)\subset D_\epsilon$.

By Lemma~\ref{lem41}, $f(D_1)$ and $f(D_2)$ are dense open subsets of $\bC$.
Thus there exists a point $w$ in the spherical disk of radius $\epsilon$
centered at $\infty$ which is contained in both $f(D_1)$ and $f(D_2)$,
say $w=f(z_j)$ where $z_j\in D_j$ for $j=1,2$.
Then $\dist(z_j,f(\gamma))\geq \pi-2\epsilon$ for $j=1,2$.
Hence the distance between $z_1$ and $z_2$ is at least $2\pi-4\epsilon$.
Since $\epsilon>0$ can be taken arbitrarily small, the conclusion follows.
\end{proof}

\begin{proposition}
\label{prop42}
If all singularities of $f$ are isolated, and there
are at least $3$ of them, then the
diameter of $(\C,\sigma)$ is at least $4\pi/3$.
\end{proposition}

\begin{proof}
Our strategy is to find a geodesic $\gamma\subset{\widetilde{\C}}$
which connects two
singularities, and such that the length of $\gamma$ is at most $2\pi/3$.
Then application of Lemma~\ref{lem41} and Lemma~\ref{lem42} will give the
required diameter estimate, similar to the argument in the proof of Proposition~\ref{prop41}.

To do this, it is sufficient to find an open metric disk in $\widetilde{\C}$
which contains no singularities, and has at least three singularities
on the boundary. Indeed, since the length of the boundary of any
spherical disk is at most $2\pi$, there will be two singularities
with the distance between them at most $2\pi/3$, and the shortest
curve between them (which exists since $\widetilde{\C}$ is complete)
will contain a geodesic arc of length at most $2\pi/3$
connecting two singularities.

It remains to show that such a disk can be found unless $f$ is an exponential.
We start with some point $w^*\in\bC$ which has an $f$-preimage $z^*\in \C$.
Let $\phi$ be the germ of $f^{-1}$ such that $\phi(w^*)=z^*$.

Let $D$ be the open spherical disk of the largest radius centered at $w^*$ to which $\phi$
has an analytic continuation. Then $\partial D$ contains
a singularity $w_1$ of $\phi$. If this singularity is unique,
then using the assumption that singularities are isolated,
we can cover $\partial D$ by finitely many disks, of which only one
contains $w_1$. Then $\phi$ has an analytic continuation to a larger
disk $D'$ such that $D\subset D'$ and $w_1\in\partial D'$. 
Among such disks $D'$ there is one with 
 maximal spherical radius, and its boundary must contain
at least two singularities, $w_1$ and $w_2$.
We denote this maximal disk $D'$ by $D_0$, and its center by
$w_0$. Let $\phi_0$ be the germ at $w_0$ obtained by the analytic
continuation of $\phi$ that we just described.

If $D_0$ has at least three singularities on the boundary, we are finished.

Otherwise, consider the curve $\beta$ passing through $w_0$
such that the points on this curve are at equal
distance from the two singularities. 
This curve is a great circle.
We assume that 
the curve is parameterized by the arc length and that $\beta(0)=w_0$.

Since the disk $D_0$ contains only two singularities $w_1$ and $w_2$ on
the boundary, we can cover the boundary by finitely many disks
of which only two contain singularities. Then there is a one-parametric
family of disks $D_t$ centered at $\beta(t)$ whose radii
are equal to the distances from $\beta(t)$ to $w_1$ and $w_2$.
Since we assumed that the only singularities on $\partial D_0$
are $w_1$ and $w_2$, the germ $\phi_0$ admits an immediate analytic
continuation from $D_0$ to $D_t$ with small $t$.

Now consider the supremum and infimum of the values of $t$
for which this analytic continuation is possible. If either of them is finite,
we obtain a disk with three singularities on the boundary.

If an analytic continuation is possible to all disks $D_t$ for $t\in\R$,
we will show that $f$ is in fact a universal covering of $\bC\backslash\{ w_1,w_2\}$, that is, $f$ is an exponential function.

To prove this last statement we consider a curve $\delta\colon\R\to\bC\setminus\{w_1,w_2\}$
with $\delta(0)=w_0$.
We ``project'' this curve $\delta$ onto the curve $\beta$ as follows:
For every $t\in\R$ there exists a unique circle which contains $w_1,w_2$ and
$\delta(t)$ and intersects the circle $\beta$ orthogonally.
(This is easy to see by applying a linear-fractional transformation
which sends the circle $\beta$ to the equator of the sphere, and sends the points
$w_1$ and $w_2$ to the poles.) The intersection of this circle with the circle
$\beta$ is the projection of $\delta(t)$. 
We thus find a continuous function $g\colon \C\to\C$ such that $\delta(t)$ 
projects to $\beta(g(t))$.
It is clear that the disk $D_{g(t)}$ contains $\delta(t)$. 
Since $\phi_0$ can be continued analytically along $\beta$, this shows that an
analytic continuation of $\phi_0$ along $\delta$ is also possible.
This completes the proof of Proposition~\ref{prop42} and Theorem \ref{thm2}.
\end{proof}

\begin{example} We construct a locally univalent meromorphic function for which
the diameter of the plane with respect to pull-back metric is infinite.

We define $f$ using a line complex \cite[Chap.\ VII]{Goldberg2008}.
Let $a,b,c,d$ be four distinct points in the Riemann sphere $\bC$.
We call them the base points.
Consider the cell decomposition $Y$ of $\bC$ shown in Fig.~\ref{line-complex1} (right).
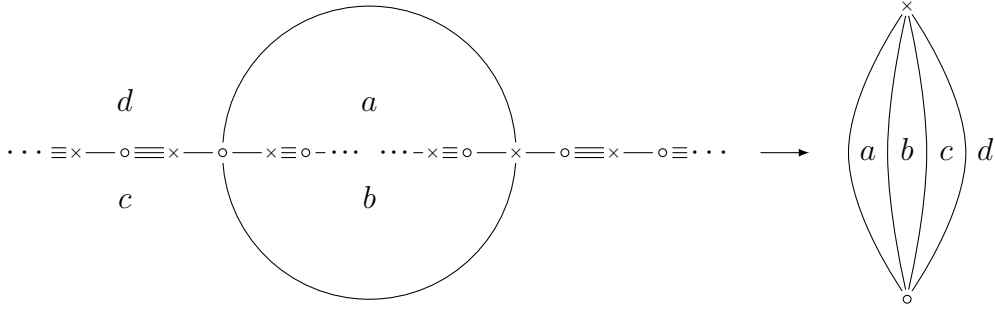
\begin{figure}[!htb]
\centering
\begin{tikzpicture}[scale=0.65,>=latex](-5,-5)(5,5)
\draw[->] (5,0) -- (6,0);
\node at (8,3) {\tiny $\times$};
\draw (8,-3) circle (0.08);
\draw plot [smooth, tension=0.8] coordinates { (7.97,2.8) (7.6,0) (7.97,-2.8) };
\draw plot [smooth, tension=0.8] coordinates { (7.9,2.83) (6.8,0) (7.90,-2.83) };
\draw plot [smooth, tension=0.8] coordinates { (8.03,2.8) (8.4,0) (8.03,-2.8) };
\draw plot [smooth, tension=0.8] coordinates { (8.1,2.83) (9.2,0) (8.10,-2.83) };
\node[anchor=mid] at (7.2,0) {$a$};
\node[anchor=mid] at (8.0,0) {$b$};
\node[anchor=mid] at (8.8,0) {$c$};
\node[anchor=mid] at (9.6,0) {$d$};
\node at (-10,0) {$\cdots$};
\node at (4,0) {$\cdots$};
\node at (-2.5,0) {$\Compactcdots$};
\node at (-3.5,0) {$\Compactcdots$};
\draw (-3,0) circle (3.0);
\filldraw[white] (0,0) circle (0.2);
\filldraw[white] (-6,0) circle (0.2);
\foreach \x in {-8,-6,-4.3,-1,1,3}
\draw (\x,0) circle (0.08);
\foreach \x in {-9,-7,-5,-1.7,0,2}
\node at (\x,0) {\tiny $\times$};
\foreach \x in {-9,-8,-7,-6,-1,0,1,2}
\draw[-] (\x+0.2,0) -- (\x+0.8,0);
\foreach \x in {-9.7,-5,-1.7,3}
{
\draw[-] (\x+0.2,0) -- (\x+0.5,0);
\draw[-] (\x+0.2,-0.1) -- (\x+0.5,-0.1);
\draw[-] (\x+0.2,0.1) -- (\x+0.5,0.1);
}
\foreach \x in {-8,1}
{
\draw[-] (\x+0.2,-0.1) -- (\x+0.8,-0.1);
\draw[-] (\x+0.2,0.1) -- (\x+0.8,0.1);
}
\foreach \x in {-4.3,-2.3}
\draw[-] (\x+0.2,0) -- (\x+0.4,0);
\node[anchor=mid] at (-3,1) {$a$};
\node[anchor=mid] at (-3,-1) {$b$};
\node[anchor=mid] at (-8,-1) {$c$};
\node[anchor=mid] at (-8,1) {$d$};
\end{tikzpicture}
\caption{Line complex of $G$.}
\label{line-complex1}
\end{figure}
It consists of two vertices $\times$ and $\circ$, four edges and four faces,
each face containing exactly one point of the set $\{ a,b,c,d\}$.
We label the faces by the base points they contain, and denote them
by $D_a,D_b,D_c,D_d$. If $f$ is a local homeomorphism $\C\to\bC$ 
whose asymptotic values are contained in the set $\{ a,b,c,d\}$
then the preimage $X=f^{-1}(Y)$ is a partition of the plane into vertices, edges and faces.
We label the vertices and faces of this partition by the same labels
as their images. 

Two such partitions are considered equivalent if they can be mapped to each other
by a homeomorphism of the plane. Two local homeomorphisms $f_1$ and $f_2$,
whose asymptotic values are contained in the set of base points,  
with equivalent partitions, satisfy $f_1=f_2\circ\phi$ where $\phi$
is a homeomorphism. A partition $X$ is completely determined
by its $1$-skeleton which is called the {\em line complex}. This is
a bipartite graph embedded in the plane whose vertices have the same
degree, equal to the number of base points.

The same construction can be made for a local homeomorphism from 
$\C^*=\C\backslash\{0\}$ or $\bC \backslash\{0\}$ to $\bC$.
When drawing a line complex, we usually do not draw the true preimage $f^{-1}(Y)$,
but an equivalent graph.

We suppose for simplicity that $\{ a,b,c,d\}\subset\C$, and consider the function $$g_1(z)=\frac{b\exp(z)-a}{\exp(z)-1}$$ which is
a universal covering of $\bC\backslash\{ a,b\}$ by $\C$. Its 
line complex consists of a chain infinite in both directions of the form
$$\cdots -\circ\equiv \times-\circ\equiv \times-\cdots .$$
Let $B$ be the region which is the union of $D_c$ and $D_d$ and the edge
between them. This region has infinitely many bounded preimages under $g_1$.
We choose one of them and call it $B_1$. 

Similarly, the function $$g_2(z)=\frac{d\exp(1/z)-c}{\exp(1/z)-1}$$
performs a universal covering of $\bC\backslash\{ c,d\}$ by
the punctured sphere $\bC\backslash\{0\}$. Its line complex is similar
to that of $g_1$, and we choose a component $B_2$ of the preimage
$g_2^{-1}(\bC\backslash B)$.

Since $g_1$ and $g_2$ map $\partial B_1$ and $\partial B_2\backslash\{0\}$
homeomorphically on the same curve $\partial B$ (with opposite orientations),
we can glue the restriction of $g_1$ on $\C\backslash B_1$
with the restriction of $g_1$ on $\C\backslash(\{0\}\cup B_2)$,
along a homeomorphism  $\psi$ 
between the boundary circles of these punctured disks,
such that $g_1=g_2\circ\psi$ on $\partial B_1$.
Since the homeomorphism $\psi$ is smooth, it has a quasiconformal extension
to a quasiconformal homeomorphism $B_1\to\bC\backslash B_2$,
which we denote by the
same letter. We can arrange that $\psi(0)=0$.
Thus we obtain a quasiregular local homeomorphism 
$$g(z)=\left\{\begin{array}{ll} g_1(z),& z\in\C\backslash B_1,\\
g_2(\psi(z)),& z\in B_1\backslash\{0\}\end{array}\right..$$

The line complex of $g$ in $\C^*$ is shown in Fig.~\ref{line-complex1} (left).
Since $g$ is quasiregular, there is a homeomorphism $\phi$ such that
$G=g\circ\phi$ is meromorphic in $\C^*$.

Next we consider the entire function $F(z)=G(\exp(iz))$. The line complex
of $F$ is shown in Fig.~\ref{line-complex2}.
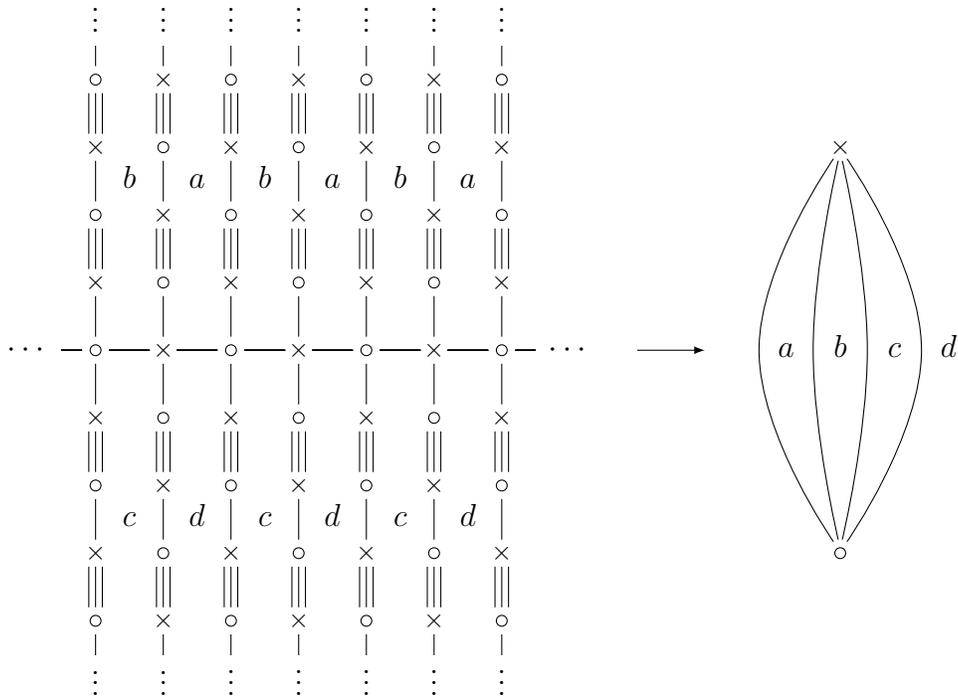
\begin{figure}[!htb]
\centering
\begin{tikzpicture}[scale=0.9,>=latex](-5,-5)(5,5)
\draw[->] (5,0) -- (6,0);
\node at (8,3) {\footnotesize $\times$};
\draw (8,-3) circle (0.08);
\draw plot [smooth, tension=0.8] coordinates { (7.97,2.8) (7.6,0) (7.97,-2.8) };
\draw plot [smooth, tension=0.8] coordinates { (7.9,2.83) (6.8,0) (7.90,-2.83) };
\draw plot [smooth, tension=0.8] coordinates { (8.03,2.8) (8.4,0) (8.03,-2.8) };
\draw plot [smooth, tension=0.8] coordinates { (8.1,2.83) (9.2,0) (8.10,-2.83) };
\node[anchor=mid] at (7.2,0) {$a$};
\node[anchor=mid] at (8.0,0) {$b$};
\node[anchor=mid] at (8.8,0) {$c$};
\node[anchor=mid] at (9.6,0) {$d$};
\foreach \x in {-2,0,2}
{
\node[anchor=mid] at (\x+0.5,2.5) {$a$};
\node[anchor=mid] at (\x-0.5,2.5) {$b$};
\node[anchor=mid] at (\x-0.5,-2.5) {$c$};
\node[anchor=mid] at (\x+0.5,-2.5) {$d$};
}
\node at (-4,0) {$\cdots$};
\node at (4,0) {$\cdots$};
\foreach \x in {-3,-2,...,3}
\foreach \y in {-4.8,5}
{
\node at (\x,\y) {$\vdots$};
}
\clip (-3.5,-4.5) rectangle (3.5,4.5);
\foreach \x in {-5,-3,...,5}
\foreach \y in {-5,-3,...,5}
{
\draw (\x,\y+1) circle (0.08);
\node at (\x,\y) {\footnotesize $\times$};
\draw (\x+1,\y) circle (0.08);
\node at (\x+1,\y+1) {\footnotesize $\times$};
\draw[-] (\x+0.2,0) -- (\x+0.8,0);
\draw[-] (\x+1.2,0) -- (\x+1.8,0);
}
\foreach \x in {-5,-4,...,5}
\foreach \y in {-5,-3,-1,0,2,4,6}
{
\draw[-] (\x,\y+0.2) -- (\x,\y+0.8);
}
\foreach \x in {-5,-4,...,5}
\foreach \y in {-6,-4,-2,1,3,5}
{
\draw[-] (\x,\y+0.2) -- (\x,\y+0.8);
\draw[-] (\x+0.1,\y+0.2) -- (\x+0.1,\y+0.8);
\draw[-] (\x-0.1,\y+0.2) -- (\x-0.1,\y+0.8);
}
\end{tikzpicture}
\caption{Line complex of $F$.}
\label{line-complex2}
\end{figure}
It remains to show that the pull-back of the spherical metric via $F$ has
infinite diameter. To do this we consider two simple curves in $\bC$:

A curve $A$ from $a$ to $d$ which is contained in the union of faces
$D_a$ and $D_d$ of $Y$ with their common boundary edge, and
a curve
$B$ from $b$ to $c$ which is contained in the union of faces $D_b$ and $D_c$
of $Y$ with their common boundary edge.

Evidently these curves have disjoint closures in $\bC$.
Each of these curves has infinitely many disjoint $F$-preimages which are
curves beginning and ending at $\infty$. Let us call these
preimages $\alpha_j$ and $\beta_j$, $j\in\Z$, and assume that they are
enumerated in the natural order, so that $\alpha_k$ separates
all $\alpha_j,\beta_j$ with $j<k$ from
$\beta_k$ and all $\alpha_j,\beta_j$ with $j>k$.

Now consider two points $p$ and $q$ in $\C$ which are separated by $2N$ curves
$\alpha_j$ and $\beta_j$. Let $\gamma$ be any curve with endpoints $p$ and $q$.
Then the image $F(\gamma)$ must hit $A$ and $B$ alternately, at least $N$
times each, so the length of this image and of $\gamma$ itself is at least
$(2N-1)\delta$ where $\delta$ is the distance between $A$ and $B$.
\end{example}

\begin{remark} One can obtain an explicit representation of the function $F$.
It can be shown that $F$ is a ratio of two solutions of the Mathieu equation
$$w''+(\cos(z/2)+\lambda)w=0,$$
where $\lambda$ is subject to the condition that this ratio has period $2\pi$.
\end{remark}

\begin{ack}
We thank Qinfeng Li whose questions stimulated this paper.
We also thank the referee for valuable comments.
\end{ack}

\end{document}